\documentclass[12pt,reqno]{amsart}
\usepackage{tikz}
\usepackage{graphicx}
\usepackage{psfrag}
\usepackage{mathrsfs}
\usepackage{color}
\usepackage{amsmath,amssymb}

\usepackage[active]{srcltx}

\numberwithin{equation}{section}

\newtheorem{maintheorema}{Theorem}

\newtheorem{theorem}{Theorem}[section]

\newtheorem{corollary}[theorem]{Corollary}
\newtheorem{proposition}[theorem]{Proposition}
\newtheorem{lemma}[theorem]{Lemma}
\newtheorem{definition}[theorem]{Definition}

\theoremstyle{remark}
\newtheorem{remark}[theorem]{Remark}
\newtheorem{problem}{Problem}

\title[Mono-atomic disintegration and Lyapunov exponents for DA ]{Mono-atomic disintegration and Lyapunov exponents for derived from Anosov diffeomorphisms}
\author{G. Ponce}
\address{Departamento de Matem\'atica,
  ICMC-USP S\~{a}o Carlos-SP, Brazil.}
  \email{gaponce@icmc.usp.br}
\author{A. Tahzibi} 
\address{Departamento de Matem\'atica,
  ICMC-USP S\~{a}o Carlos-SP, Brazil.}
\email{tahzibi@icmc.usp.br}
\author{R. Var\~{a}o} 
\address{Departamento de Matem\'atica,
  ICMC-USP S\~{a}o Carlos-SP, Brazil {\&  Department of Mathematics, University of Chicago, Chicago, USA.}}
\email{regisvarao@icmc.usp.br}

\date{}                                          % Activate to display a given date or no date
\thanks{G. Ponce is enjoying a Doctoral scholarship of FAPESP process \# 2009/16792-8,  \# 2012/14620-8, and is grateful for the hospitality of Penn State University during the final writing of this work . A. Tahzibi had the support of CNPq and FAPESP. R. Var\~ao got financial support from FAPESP process \# 2011/21214-3  and  \# 2012/06553-9, as well as with the hospitality of the University of Chicago during the final writing of this work.}

\begin{document}
\maketitle

\begin{abstract}

In this paper we mainly address the problem of disintegration of  Lebesgue measure and measure of maximal entropy along the central foliation of  (conservative) Derived from Anosov (DA) diffeomorphisms. We prove that for accessible 
DA diffeomorphisms of $\mathbb{T}^3$, atomic disintegration has the peculiarity of being mono-atomic (one atom per leaf). We further provide open and non-empty condition for the existence of atomic disintegration.  Finally, we prove some new relations between Lyapunov exponents of DA diffeomorphisms and their linearization.

%For conservative Derived from Anosov (DA) diffeomorphims on $\mathbb T^3$, we show that if the center leaves disintegrate on atomic measures, than in fact this has to be one atom per leaf. On the other hand, until now, it was not clear if such DA diffeomorphims could exist. We give an open (non-empty) condition for the existence of atomic disintegration. Such DA's are satisfied by a class found by Ponce-Tahzibi \cite{PT}. The atomic behavior shows a feature for DA diffeomorphims that is not, so far, shared with any other known partially hyperbolic diffeomorphims on dimension three, it admits all three disintegration: Lebesgue, atomic, and, by a recent result of R. Var\~ao \cite{Va}, they can also have a disintegration which is neither Lebesgue nor atomic.

\end{abstract}

\tableofcontents

\section{Introduction}

%In this paper we address the problem of disintegration of volume on center leaves of a Derived from Anosov (DA) diffeomorphisms on $\mathbb{T}^3$ (definitions on \S \ref{subsec.ph}). 

A diffeomorphism is called partially hyperbolic if the tangent bundle admits an invariant decomposition $TM = E^s \oplus E^c \oplus E^u$, such that all unit vectors $v^{\sigma} \in E^{\sigma}_x, \sigma \in \{s, c, u\}$ for all $x \in M$ satisfy:
$$
  \|D_xf v^s \| < \|D_xf v^c\| < \| D_xf v^u\|
$$ 
and moreover $ \|D_f| E^s\| < 1$ and $\|D f^{-1} | E^{u}\| < 1.$ We call $f$  {\bf absolute} partially hyperbolic if it is partially hyperbolic and for any $x, y, z \in M$
$$
 \|D_xf v^s \| < \|D_yf v^c\| < \| D_zf v^u\|
$$ 
where $v^s, v^c, v^u$ belong respectively to $E_x^s, E^c_y$ and $E^u_z.$
In this paper when we require partial hyperbolicity, we mean absolute partial hyperbolicity.

%
%A partially hyperbolic diffeomorphism $f:M \rightarrow M$ has three $Df$-invariant subbundles namely $E^s$, $E^c$, $E^u$ where the bundles $E^s$ and $E^u$ are uniformly contracted and expanded by $Df$ respectively; $E^c$ behaves as not contracting as much as $Df$ contracts in the stable direction ($E^s$) neither expanding as much as in the unstable direction ($E^u$). 

The subbundles $E^s$ and $E^u$ integrate into $f$-invariant foliations, respectively the stable foliation, $\mathcal F^s$, and the unstable foliation, $\mathcal F^u$. These foliations have  a nice property called absolute continuity. Among different definitions for absolute continuity of foliations we choose the following one which is suitable for our purpose: A set of full volume measure on $M$ must intersect almost every leaf of $\mathcal F^s$ (or $\mathcal F^u$) in a set of  full Lebesgue measure of the leaf. Although the absolute continuity of $\mathcal F^s$ and $\mathcal F^u$ are mandatory for a (general)  $C^2$ partially hyperbolic diffeomorphism, this is not the case for the center 
foliation $\mathcal F^c$ (it is not even true that there will exist such a foliation, but by \cite{BBI09} for all absolute partially hyperbolic diffeomorphisms of $\mathbb{T}^3$ the center foliation exists). The center foliation might not be absolutely continuous, at least this is, in general, expected to happen for diffeomorphisms which preserves volume (see \cite{SX}, \cite{HP},  \cite{Gogolev1}). For many examples (some of them described below) the center foliation has atomic disintegration. In all such examples there exists a set $A \subset M$ of full volume measure, such that the intersection of $A$ with every leaf of $\mathcal F^c$ are $k$ points, where $k$ is a natural number independent of the leaf (in the ergodic context). In principle for a general partially hyperbolic difeomorphism the geometric structure of the support of disintegration measures is not clear. We do not have examples of atomic disintegration with infinitely many atoms. 

 There exist essentially three known category of partially hyperbolic diffeomorphisms on  three-dimensional manifolds (see conjecture of Pujals in \cite{BW05}). We deal here with the so called Derived from Anosov (DA) diffeomorphisms. By a DA diffeomorphism we mean a partially hyperbolic diffeomorphim $f:\mathbb T^3 \rightarrow \mathbb T^3$ such that its linearization (see \ref{prepartial} in Preliminaries) is an Anosov diffeomorphism. Observe that by \cite{Htese} the linearization of $f$ is also partially hyperbolic in the sense that it admits three invariant sub-bundles.   The other two classes of partially hyperbolic diffeomorphisms are the skew-product type and perturbations of time-one of Anosov flows (see as well Hammerlindl-Potrie \cite{HP13} for a discussion and new results). 

For the perturbation of a time-one map of the geodesic flow for a closed negatively curved surfaces (which is an Anosov flow), it was shown by A. Avila, M. Viana and A. Wilkinson \cite{AVW} that $\mathcal F^c$ has atomic disintegration or it is absolutely continuous. For a large class of skew-product diffeomorphisms, they announced that  they can prove an analogous result.

It is interesting to emphasize that (conservative) Derived from Anosov (DA) diffeomorphims on $\mathbb T^3$ show a feature that is not, so far, shared with any other known partially hyperbolic diffeomorphims on dimension three, it admits all three disintegration of volume on the center leaf, namely: Lebesgue, atomic, and, by a recent result of R. Var\~ao \cite{Va}, they can also have a disintegration which is neither Lebesgue nor atomic.

 More precisely, R. Var\~ao \cite{Va} showed that there exist Anosov diffeomorphisms with non-absolutely continuous center foliation which does not have atomic disintegration. Here we show a new behavior for DA diffeomorphisms on $\mathbb T^3$, and that is the existence of atomic disintegration (Theorem \ref{theorem:existatom}). This behavior can be verified for an open class of diffeomorphisms found by  Ponce-Tahzibi in \cite{PT}. In fact we prove  (Theorem \ref{theorem:oneatom}): if the disintegration of Lebesgue measure is atomic, then it is in fact mono-atomic, i.e there is just one atom per leaf. We should mention that the most important part of our proof is to conclude finiteness of atoms in the case of atomic disintegration. However, for general partially hyperbolic diffeomorphisms finiteness does not imply mono-atomic disintegration. See \cite{RW} and our discussion after Theorem A below for a contrast with the skew product case.\\
%but that is not the only possibility as we have shown here. In fact a DA on $\mathbb T^3$ have richier possibility of configurations of disintegration as when compared to the other two known types of partially hyperbolic diffeomorphisms on three dimensional manifolds. We have shown that in our context when atomicity occurs it must be one atom per leaf, that is equivalent as having a set of full volume measure which intersects each leaf on exactly one point. For a discussion and definitions on disintegration see \S \ref{subsec.partition}.
%The reader should check \S \ref{sec.preliminaries} for definitions and discussions on the concepts we use here, such as partially hyperbolic, absolute continuity, disintegration. Before stating more precisely our results of this first part (the second part being the maximal entropy measures and its relation to Lyapunov exponents, see \S \ref{subsec.max.entropy}), it is worth fixing some notation and standing hypothesis. 

%There will be no latter mention on the regularity of the diffeomorphisms we are working with, but we state that it is a hypothesis assumed throughout the text that all diffeomorphisms are at least $C^{1 + \alpha}$.

\textbf{Theorem \ref{theorem:oneatom}}  \textit{Let $f$ be a volume preserving accessible DA diffeomorphism on $\mathbb T^3.$ If volume has atomic disintegration on center leaves, then it has one atom per leaf.}\\
 
\begin{remark} \label{remark: nonnull} The accessibility hypothesis for partially hyperbolic diffeomorphisms with one dimensional center bundle implies  that $f^n$ is ergodic for all $n \geq 1.$ (see \cite{HHU1} or \cite{BW}.) This is the unique place that we use accessibility, so that the accessibility hypothesis can be substituted by any other  hypothesis that provides the ergodicity of all $f^n$. For instance, by Hammerlindl-Ures \cite{HU} we know that if the center exponent of $f$ is not zero almost everywhere (consequently the center exponent of $f^n$ is not zero almost everywhere, for all $n$) then it is ergodic (consequently $f^n$ is ergodic for all $n$). This observation will be important in the proof of Theorem B.
\end{remark} 
 
 The conclusion about the number of atoms is not true for skew-products in general. Let $B := A \times Id$ where   
$ A:= \left( \begin{array}{cc}
2 & 1 \\
1 & 1 \\ \end{array} \right)$. Then arbitrarily close to $B : \mathbb{T}^3 \rightarrow \mathbb{T}^3 $ there is an open set of partially hyperbolic diffeomorphisms $g$ such that $g$ is ergodic and there is an equivariant fibration $\pi: \mathbb{T}^3 \rightarrow \mathbb{T}^2$ such that the fibers are circles, $\pi \circ g = B \circ \pi$. And $g$ has positive central Lyapunov exponent, hence the central foliation is not absolutely continuous. In Ruelle and Wilkinson's paper \cite{RW}, we see that there exist $S \subset \mathbb{T}^ 3$ of full volume and $k \in \mathbb{N}$ such that $S$ meets every leaf in exactly $k$ points. In Shub and Wilkinson's example \cite{SW00} the fibers of the fibration are invariant under the action of a finite non-trivial group  and consequently in their example the number of atoms cannot be one.
 
In the next theorem we introduce a class of derived from Anosov diffeomorphisms where Theorem A can be applied. The work of Ponce, Tahzibi \cite{PT} guaranties that there exist DA diffeomorphisms satisfying these conditions.\\

\textbf{Theorem \ref{theorem:existatom}}
\textit{Let $f: \mathbb T^3 \rightarrow \mathbb T^3$ be a volume preserving, DA diffeomorphism. Suppose that its linearization $A$ has the splitting $T_AM = E^{su} \oplus E^{wu} \oplus E^s$ (su and wu represent strong unstable and weak unstable bundles.) If $f$ has $\lambda ^{c}(x) <0$ for Lebesgue almost every point $x \in \mathbb T^3$, then volume has atomic disintegration on $\mathcal F^{c}_f$, in fact the disintegrated measures  have one atom per center leaf.}\\

\begin{remark}\label{remark:proof.theorem}
The key point in the proof of Theorem \ref{theorem:existatom} is  show that we get atomic disintegration. Then, we use Theorem \ref{theorem:oneatom} (because of Remark \ref{remark: nonnull}) to obtain one atom per leaf.
\end{remark}

We look at the linearization of $f$, the linear Anosov $A$, as a partially hyperbolic for which the center leaf $\mathcal F^c_A$ is the expanding leaf $\mathcal F^{wu}_A$.

We mention that the examples of non-absolutely continuous weak foliation of Anosov diffeomorphisms was known by Saghin-Xia \cite{SX} and Baraviera-Bonatti \cite{BB} near to geodesic flows,  and by A.Gogolev near to hyperbolic automorphisms of the $3$-torus \cite{Gogolev1}. However, we are introducing examples of non-Anosov diffeomorphisms with non-absolutely continuous central foliation and prove atomic disintegration. It is not known whether the disintegration of the Lebesgue measure can be atomic in the case of Anosov diffeomorphisms.  

% Eu (Regis) omiti a proxima frase porque mudando o enunciado do Teorema A eu tive que falar logo depois que f Ã© ergodica, entÃ£o tirei a frase abaixa para nao ficar muito redundante.
%It is well worth to emphasize that although we are not assuming the ergodicity, it is the case, by a recent result of Hammerlindl-Ures \cite{HU}.  

\subsection{Maximal entropy measures and Lyapunov exponents}\label{subsec.max.entropy}

Lyapunov exponents are celebrated constants which are related to the entropy of invariant measures. In this paper  we denote by $\lambda^u(f), \lambda^c(f)$ and $\lambda^s(f)$ the Lyapunov exponents of Lebesgue measure (ergodic case).  When we are referring to the Lyapunov exponents of any other (ergodic) invariant measure $\mu$ we use the subscript $\lambda^{*}_{\mu}.$ 
The relationship between the Lyapunov exponents of a partially hyperbolic diffeomorphism and its linearization is an interesting issue. In \cite{MT}, the authors among other results proved a folkloric semi-rigidity property of unstable and stable Lyapunov exponents of  (absolute) partially hyperbolic diffeomorphisms of $\mathbb{T}^3:$

\begin{theorem}
Let $f$ be a $C^2$ conservative partially hyperbolic diffeomorphism on the $3-$torus and $A$ its linearization then

$$ \lambda^u(f,x) \leq \lambda^u(A) \; \mbox{and} \;\; \lambda^s(f,x) \geq \lambda^s(A)\; \mbox{for Lebesgue a.e. } \; x \in \mathbb{T}^3.$$
\end{theorem} 

The above theorem relates the extremal Lyapunov exponents of a non-linear and linear partially hyperbolic diffeomorphism. For the central Lyapunov exponent we do not expect such behavior. However, it is a problem in \cite{MT}:

\begin{problem} In the context of the above theorem,
suppose that $\lambda^{c}(f) > 0$ and $\mathcal{F}^c$ is (upper leafwise) absolutely continuous. Is it true that $\lambda^{c}(f) \leq \lambda^{c}(A)$ ?
\end{problem}

The answer to the above problem is positive when $f$ is an Anosov diffeomorphism (see \cite{MT, Va}).

We emphasize that the above relations between Lyapunov exponents are valid for the volume measure. For other natural invariant measures the scenario can be different. Let $\mu_f$ be the unique maximizing measure for a partially hyperbolic diffeomorphism $f$ homotopic to Anosov $A.$ In \cite{Ures},  R. Ures  proved that if $\lambda^ c(A) > 0$, i.e $A$ has weak unstable sub-bundle then $\lambda^ c_{\mu_f} \geq \lambda^ c(A).$  

In the following theorem we see that if the central Lyapunov exponent for Lebesgue measure of $f$ is strictly smaller than the central exponent of the linearization $A,$ then either the central or the unstable exponent of the maximizing measure of $f$ is strictly larger than the corresponding exponent of $A$.

Before stating our last theorem we give one more definition. A measure is called u-Gibbs if the disintegration subordinated to the unstable foliation (corresponding to all positive Lyapunov exponents) gives conditional measures equivalent to the Lebesgue on the leaf. These measures play an important role in dynamical systems, \cite{LY1}.\\

\textbf{Theorem \ref{ugibbs}}
\textit{Let $f : \mathbb T^3 \rightarrow \mathbb T^3$ be a DA diffeomorphism with $A: \mathbb T^3 \rightarrow \mathbb T^3$ its linearization. Assume that $\lambda^c(A)>0$ and $\lambda^c(f) < \lambda^c(A)$. Then $\mu_f$ is not $u$-Gibbs. Also, denoting by $\lambda^u_{\mu_f}, \lambda^c_{\mu_f}, \lambda^s_{\mu_f}$ the Lyapunov exponents of $f$ (for $\mu_f$ almost every point) we have
$$\lambda^c_{\mu_f} > \lambda^c(A) \text{ or } \lambda^u_{\mu_f} > \lambda^u(A) .$$}

We recall  the example of Ponce-Tahzibi in \cite{PT} where $f$ has negative central Lyapunov exponent and the linearization of $f$ has positive central Lyapunov exponent. The authors began with a linear Anosov diffeomorphism $A$ with splitting $E^{su} \oplus E^{wu} \oplus E^ {s}$ and after a modification they found $f$ partially hyperbolic ergodic volume preserving such that $\lambda^c(f) < 0.$ The point is that, in their construction $\lambda^u_{\mu_f} (f) =\lambda^u(A)$ and by the above theorem we conclude that $\lambda^c_{\mu_f} > \lambda^c(A).$ That is, although after perturbation the central Lyapunov exponent of Lebesgue measure drops, the central exponent of maximal entropy measure increases. We describe in some more details the construction of this example on \S \ref{subsec.pathological}.

\section{Preliminaries}\label{sec.preliminaries}

\subsection{Measurable partitions and disintegration of measures}\label{subsec.partition}

Let $(M, \mu, \mathcal B)$ be a probability space, where $M$ is a compact metric space, $\mu$ a probability measure and $\mathcal B$ the borelian $\sigma$-algebra.
Given a partition $\mathcal P$ of $M$ by measurable sets, we associate the probability space $(\mathcal P, \widetilde \mu, \widetilde{\mathcal B})$ by the following way. Let $\pi:M \rightarrow \mathcal P$ be the canonical projection, that is, $\pi$ associates a point $x$ of $M$ with the partition element of $\mathcal P$ that contains it. Then we define $\widetilde \mu := \pi_* \mu$ and $ \widetilde{\mathcal B}:= \pi_*\mathcal B$.

\begin{definition} \label{definition:conditionalmeasure}
 Given a partition $\mathcal P$. A family $\{\mu_P\}_{p \in \mathcal P} $ is a \textit{system of conditional measures} for $\mu$ (with respect to $\mathcal P$) if
\begin{itemize}
 \item[i)] given $\phi \in C^0(M)$, then $P \mapsto \int \phi \mu_P$ is measurable;
\item[ii)] $\mu_P(P)=1$ $\widetilde \mu$-a.e.;
\item[iii)] if $\phi \in C^0(M)$, then $\displaystyle{ \int_M \phi d\mu = \int_{\mathcal P}\left(\int_P \phi d\mu_P \right)d\widetilde \mu }$.
\end{itemize}
\end{definition}

When it is clear which partition we are referring to, we say that the family $\{\mu_P\}$ \textit{disintegrates} the measure $\mu$.  

\begin{proposition}
 If $\{\mu_P\}$ and $\{\nu_P\}$ are conditional measures that disintegrate $\mu$, then $\mu_P = \nu_P$ $\widetilde \mu$-a.e.
\end{proposition}

\begin{corollary} 
 If $T:M \rightarrow M$ preserves a probability $\mu$ and the partition $\mathcal P$, then  $T_*\mu_P = \mu_P$ $\widetilde \mu$-a.e.
\end{corollary}
\begin{proof}
 It follows from the fact that $\{T_*\mu_P\}_{P \in \mathcal P}$ is also a disintegration of $\mu$.
\end{proof}

\begin{definition}
We say that a partition $\mathcal P$ is measurable with respect to $\mu$ if there exist a measurable family $\{A_i\}_{i \in \mathbb N}$ and a measurable set $F$ of full measure such that 
if $B \in \mathcal P$, then there exists a sequence $\{B_i\}$, where $B_i \in \{A_i, A_i^c \}$ such that $B \cap F = \bigcap_i B_i \cap F$.
\end{definition}
\begin{proposition} \label{prop:compact}
Let $(M,\mathcal B, \mu)$ a probability space where $M$ is a compact metric space and $\mathcal B$ is the Borel sigma-algebra. If $\mathcal P$ is a continuous foliation of $M$ by compact measurable sets, then $\mathcal P$ is a measurable partition. 
\end{proposition}

\begin{theorem}[Rokhlin's disintegration] \label{teo:rokhlin} %\cite{viana.rokhlin}] 
 Let $\mathcal P$ be a measurable partition of a compact metric space $M$ and $\mu$ a borelian probability. Then there exists a disintegration by conditional measures for $\mu$.
\end{theorem}
In general the partition by the leaves of a foliation may be non-measurable. It is for instance the case for the stable and unstable foliations of a linear Anosov diffeomorphism. Therefore, by disintegration of a measure along the leaves of a foliation we mean the disintegration on compact foliated boxes. In principle, the conditional measures depend on the foliated boxes, however, two different foliated boxes induce proportional conditional measures. See \cite{AVW} for a discussion. We define absolute continuity of foliations as follows:  

\begin{definition}
 We say that a foliation $\mathcal F$ is absolutely continuous if for any foliated box, the disintegration of volume on the segment leaves have conditional measures equivalent to the Lebesgue measure on the leaf.
\end{definition}

\begin{definition}
We say that a foliation $\mathcal F$ has atomic disintegration with respect to a measure $\mu$ if the conditional measures on any foliated box are a sum of Dirac measures. Note that this is equivalent to saying that there is a set of $\mu$-full measure that intersects each leaf on a discrete set.
\end{definition}

Although the disintegration of a measure along a general foliation is defined in compact foliated boxes, it makes sense to say that the foliation $\mathcal F$ has a quantity $k_0 \in \mathbb N$  atoms per leaf. The meaning of ``per leaf'' should always be understood as a generic leaf, i.e. almost every leaf. That means that there is a set $A$ of $\mu$-full measure which intersects a generic leaf on exactly $k_0$ points. Let's see that this implies atomic disintegration. Definition \ref{definition:conditionalmeasure} shows that it only make sense to talk about conditional measures from the generic point of view, hence when restricted to a foliated box $\mathfrak{ B}$, the set $A \cap \mathfrak{B}$ has $\mu$-full measure on $\mathfrak{B}$, therefore the support of the conditional measure disintegrated on $\mathfrak{B}$ must be contained on the set $A$. This implies atomic disintegration.

It is well worth to remark that  the weight of an atom for a conditional measure naturally depends on the foliated box, but a point $x$ is atom independent of the foliated box where we disintegrate a measure.   

\subsection{ Preliminaries on Partial Hyperbolicity}\label{subsec.ph} \label{prepartial}

%\begin{definition}
% A diffeomorphism $f$ of a compact Riemannian manifold $M$ is called partially hyperbolic  if there are constants $\lambda < \hat{\gamma} < 1 < \gamma < \mu$ and $C > 1$ and a $Df$ -invariant splitting of $ T M = E^u (x) \oplus E^c (x) \oplus E^s (x)$ where
%\begin{eqnarray*}
% \frac{1}{C}\mu^n||v|| < & ||Df^nv||, &   \quad \quad  \quad  \quad \quad  \; \; v \in E^s_x - \{0\}; \\
% \frac{1}{C}\hat{\gamma}^n||v|| < &||Df^nv||& < C \gamma^n ||v||, \; \;  v \in E^c_x - \{0\}; \\
%		&||Df^nv||&< C \lambda^n ||v||,   \; \;  v \in E^u_x - \{0\}.
%\end{eqnarray*}
%
%\end{definition}
%

If $f: M \rightarrow M$ is a partially hyperbolic diffeomorphism with
$$T_xM = E^s(x) \oplus E^c(x) \oplus E^u(x)$$
the bundles $E^s$ and $E^u$ are tangent to invariant foliations $\mathcal{F}^s$ and $\mathcal{F}^u.$ $f$ is called {\bf accessible} if for any two points $x$ and $y$ there is a piecewise smooth curve connecting $x$ to $y$ and tangent to $E^s \cup E^u.$ If the central bundle is one dimensional and $f$ is volume preserving then accessibility of $f$ implies that all iterates $f^n$ are ergodic (\cite{BW}, \cite{HHU1}).  

We define the Lyapunov exponents of $f$ by
$$\lambda^{\tau}(x) := \lim_{n\rightarrow \infty}\frac{1}{n} \log ||Df^n(x) \cdot v||$$

where $v  \in E^{\tau}$ and $\tau \in \{s,c,u\}$. If $f$ is ergodic, the stable, center and unstable Lyapunov exponents are constant almost everywhere. Otherwise they have no reason to be constant in the general case. \\

Any $A \in SL(3, \mathbb{Z})$ with at least one eigenvalue with norm larger than one, induces a linear partially hyperbolic diffeomorphism on $\mathbb{T}^3.$ Conversely for any partially hyperbolic diffeomorphism $f$ on $\mathbb{T}^3,$ there exist a unique linear diffeomorphism $A,$ such that $A$ induces the same automorphism as $f$ on the fundamental group $\pi_1(\mathbb{T}^3).$    

Let $f: \mathbb{T}^3 \rightarrow \mathbb{T}^{3}$ be a partially hyperbolic diffeomorphism. Consider $f_* : \mathbb{Z}^3 \rightarrow \mathbb{Z}^3$ the action of $f$ on the fundamental group of $\mathbb{T}^3.$ $f_*$ can be extended to $\mathbb{R}^3$ and the extension is the lift of  a unique linear automorphism $A :  \mathbb{T}^{3} \rightarrow \mathbb{T}^{3}$ which is called the linearization of $f.$ It can be proved that $A$ is a partially hyperbolic automorphism of torus (\cite{BBI}). We will say that $f$ is {\bf derived from Anosov} (DA diffeomorphism) if  its linearization $f_*$ is an Anosov diffeomorphism. 

Let $f$ be DA diffeomorphism defined  as above, then by \cite{Franks} we know that $f$ is semi-conjugated to its linearization by a function $h: \mathbb{T}^3 \rightarrow \mathbb{T}^3, h \circ f = A \circ h$. It follows from \cite{Ures} that $\mathcal{F}^c (A) = h(\mathcal{F}^c (f)).$ Moreover, there exists a constant $K  \in \mathbb{R}$ such that  if $\tilde{h} : \mathbb{R}^3 \rightarrow \mathbb{R}^3 $ denotes the lift of $h$ to $\mathbb{R}^3$ we have $\|\tilde h(x) - x\| \leq K$ for all $x \in \mathbb{R}^3.$

It is not difficult to see that in large scale $f$ and $A$ behaves similarly (see \cite{Htese}, lemma 3.6). More precisely, for each $k \in \mathbb{Z}$ and $C > 1$ there is an $M > 0 $ such that for all $x, y \in \mathbb{R}^3$,
\begin{equation}
\label{large-scale}
\|x - y\| > M \Rightarrow \frac{1}{C} < \frac{\|\tilde{f}^k(x) - \tilde{f}^k(y)\|}{\|A^k(x) - A^k(y)\|} < C.
\end{equation}
where $\tilde{f}: \mathbb{R}^3 \rightarrow \mathbb{R}^3$ is the lift of $f$ to $\mathbb{R}^3.$

% \begin{lemma} [\cite{MT}] \label{linalg}  \marginpar{Ali : usamos este lema algum lugar? \\ G:Acho que n\~ao Ali.}
 %Let $f : \mathbb{T}^3 \rightarrow \mathbb{T}^3$ be a partially hyperbolic diffeomorphism and $A: \mathbb{T}^3 \rightarrow \mathbb{T}^3$ the linearization of $f.$ For all $n \in \mathbb{Z}$ and $\epsilon > 0$ there exists $M$ such that for $x, y$ with $y \in \mathcal{F}^{\sigma}_x$ and $||x -  y||> M$ then
 %$$
  % (1 - \varepsilon)e^{n\lambda^{\sigma} (A) } ||y -x|| \leq \|A^n(x) - A^n(y)\| \leq (1 + \varepsilon)e^{n\lambda^{\sigma} (A) } ||y -x||
 %$$
%where $\lambda^{\sigma} (A)$ is the Lyapunov exponent of $A$ corresponding to $E^{\sigma}$ and $\sigma \in \{s, c, u\}.$
% \end{lemma}

\begin{definition} \label{qi}
A foliation $\mathcal F$ defined on a manifold $M$  is quasi-isometric if the lift $\widetilde{ \mathcal F}$ of $\mathcal F$ to the universal cover of $M$ has the folowing property: There exist positive constant $Q$
such that for all $x, y$ in a common leaf of $\widetilde{ \mathcal F}$ we have
$$d_{\widetilde{\mathcal{F}}}(x, y) \leq Q || x - y||,$$
where $d_{\widetilde{\mathcal{F}}}$ denotes the riemannian metric on $\widetilde{\mathcal{F}}$ and $\|x-y\|$ is the distance on the ambient manifold of the foliation.
\end{definition}

For absolute partially hyperbolic diffeomorphisms on $\mathbb{T}^3$ the stable, unstable and central foliations are quasi isometric in the universal covering $\mathbb{R}^3$.

\subsection{``Pathological" example}\label{subsec.pathological}
As we remarked before, in theorem  \ref{theorem:existatom} one of the hypothesis is that the center Lyapunov exponent of the diffeomorphism $f$ and of its linearization $A$ have opposite sign. Since we have center leaf conjugacy between $f$ and $A$ and since in large scale the behavior of the center leafs  is similar (see (\ref{large-scale})), this hypothesis would imply that the asymptotic growth of the center leaves (which is a local issue) and global behavior of the center leafs of $f$ are opposite. In \cite{PT}, the authors constructed an example to show that this kind of phenomena occurs in an open set of partially hyperbolic dynamics and we briefly describe it here.

Start with the family of linear Anosov diffeomorphisms $f_k: \mathbb T^3 \rightarrow \mathbb T^3$ induced by the integer matrices: 
$$A_k=\left( \begin{array}{ccc}
0 & 0 & 1 \\
0 & 1 & -1 \\
-1 & -1 & k \end{array} \right), k \in \mathbb N. $$

This family of Anosov diffeomorphisms has two important characteristics that justify this choice.  Denote by $\lambda^s_k, \lambda^c_k, \lambda^u_k$ the three eigenvalues of $A_k$ with $\lambda^s_c < \lambda^c_k<\lambda^u_k$ and, for each $k$, denote by $E^s_k, E^c_k, E^u_k$ the stable, central and unstable fiber bundles with respect to $A_k$. Then an easy calculation shows that
$$\lambda^s_k \rightarrow 0, \lambda^c_k \rightarrow 1, \lambda^u_k \rightarrow \infty,$$
as $k \rightarrow \infty$.

Using a Baraviera-Bonatti \cite{BB} local perturbation method, for large $k$ the authors managed to construct small perturbation of $f_k$ and obtain partially hyperbolic diffeomorphisms $g_k$ such that the central Lyapunov exponent of $g_k$ is positive.

By taking the family $g_k^{-1}$ we obtain partially hyperbolic diffeomorphisms with negative center exponent and isotopic to  Anosov diffeomorphism with weak expanding subbundle. In fact any $f:=g_k^{-1}$ satisfy the desired properties.
By construction of the perturbation, it comes out that  the center-stable bundle of $f$ coincides with the sum of  stable and weak unstable bundle of Anosov linearizations $f_* = A$: $E^{cu}f (x) =  E_A^{wu}(x) \oplus E_A^{s}(x)$. Moreover, for any $x \in \mathbb{T}^3$ we have $\log J^{cs}f = \log J^sA + \log J^{wu}(A).$ As both $f$ and $A$ are volume preserving we conclude that $\displaystyle{\int}
\log J^u(f, x) d\mu = \lambda^u(A)$ and by ergodicity $\lambda_f^u(x) = \lambda^u(A)$ for $\mu$ almost every $x$ where $\mu$ is any invariant probability measure. 
 
%{\color{red} (Confiram para ver se no fiz besteira em algum passo. Isso foi o que comentamos numa conversa ano passado.

%proof: \\
%In the case $E^{cs}_f(x)=E^{cs}_A(x)$ then 

%$$\lambda^c_f+\lambda^s_f = \int \log |\det ( Df(x)|E^{cs}_f(x) )| d\mu = \int \log |\det (DA(x)|E^{cs}_A(x))| d\mu = \lambda^s+\lambda^c$$

%where the second equality comes from the fact that $f= A\circ h$ where $h$ preserves the $u$-coordinate (so $det Dh(x)|E^{cs} = 1$) and the last equality occurs because $\det (A|E^{cs}_A(x)) = \lambda^c \cdot \lambda^s$, for every $x\in \mathbb T^3$.

%In particular, 
%$$0=\int \log |\det Df(x)| d\mu =\int \log |\det (Df(x)|E^u_f(x)) |d\mu + \int \log |\det (Df(x)|E^{cs}_f(x))| d\mu =$$
%$$= \lambda^u_f+\lambda^s+\lambda^c.$$

%Then we conclude that $\lambda^u_f = \lambda^u$, which implies $\lambda^c_f > \lambda^c$ as we wanted to show.}

 This family of diffeomorphisms fulfills the hypothesis required in Theorems \ref{theorem:existatom} and \ref{ugibbs}.

\section{Proof of results}

\begin{maintheorema}\label{theorem:oneatom} Let $f$ be a volume preserving accessible DA diffeomorphism on $\mathbb T^3.$ If the volume has atomic disintegration on the center leaves, then it has one atom per leaf.
\end{maintheorema}
\begin{proof}
Let $h: \mathbb T^3 \rightarrow \mathbb T^3$ be the semi-conjugacy between $f$ and its linearization $A$, hence $ h\circ f = A \circ h$. We can assume that  $A$ has two eigenvalues larger than one, otherwise we work with $f^{-1}$.
 
 Let $\{R_i\}$ be a Markov partition for $A$, and define $\widetilde R_i := h^{-1}(R_i)$. We claim that
 \begin{equation} \label{interior}
 Vol \left( \bigcup_i int \;\widetilde R_i \right) =1.
 \end{equation}

Indeed, first look at the center direction of $A$. For simplicity we consider the center direction as a vertical foliation. This means that the rectangle $R_i$ has two types of boundaries, the one coming from the extremes of the center foliation and the lateral ones. We call $\partial_c R_{i}$ the boundary coming from these extremes of the center foliation, i.e 
$$
 \partial_c R_i = \bigcup_{x \in R_i} \partial (\mathcal{F}^c_{x} \cap R_i).
$$ Since $h$ takes center leaves to center leaves, we conclude that the respective boundary for the $\widetilde R_i$ sets is $\partial_c \widetilde R_i = h^{-1}(\partial_c R_{i})$, and since $\bigcup_i \partial_c R_{i}$ is an $A$-invariant set, $\bigcup_i \partial_c \widetilde R_i $ is an $f$-invariant set. By ergodicity of $f$ it follows that $\bigcup_i \partial_c \widetilde R_i $ has zero or full measure. Since the volume of the interior cannot be zero, then the volume of $\bigcup_i \partial_c \widetilde R_i $ cannot be one. Therefore it has zero measure.  

%Since $h$ is from a bounded distance of identity, the sets $\widetilde R_i$ are indeed bounded. 
By (\ref{interior}) we can consider the partition $\widetilde {\mathcal P} = \{ \mathcal F^c_{R(x)} : x \in \widetilde R_i \text{ for some } i\}$ where $\mathcal F^c_{R(x)}$ denotes the connected component of $\mathcal F^c_f(x) \cap \widetilde R(x)$ which contains $x$. Thus we can consider the Rokhlin disintegration of volume on the partition $\widetilde {\mathcal P}$.  Denote this system of measures by $\{m_x\}$, so that each $m_x$ is supported in $\mathcal F^c_f(x)$.

% If $x$ belongs to the interior of some $\widetilde R_i$, $m_x$ will denote the Rokhlin disintegration of volume on the $\widetilde R_i$ that contains $x$, where the partition is given by center plaques trough $\widetilde R_i$. We may also denote $\widetilde R (x)$ for one of the $\{\widetilde R_i\}$ which contains $x$ and $\mathcal F^c_{R(x)}$ means the connected component of $\mathcal F^c_f(x) \cap \widetilde R(x)$ which contains $x$.
 
\begin{lemma}
 There is a natural number  $\alpha_0 \in \mathbb N$, such that for almost every point, $\mathcal F^c_{R(x)}$ contains exactly $\alpha_0$ atoms. 
\end{lemma}
\begin{proof}
The semi-conjugacy $h$ sends center leaves of $f$ to center leaves of $A$. Also, the points of the interior of the $\widetilde R_i$ satisfy that $f(\mathcal F^c_{R(x)}) \supset \mathcal F^c_{R(f(x))}$, which just comes from the Markov property of the rectangles $R_i$. This implies that
$$f_*m_x \leq m_{f(x)}$$
on $\mathcal F^c_{R(f(x))}$. 
 Given any $\delta \geq 0$ consider the set $A_\delta =\{ x \in \mathbb{T}^3 \; | \; m_{x}(\{x\}) > \delta \}$, that is, the set of atoms with weight at least $\delta$. 
 If $x\in A_{\delta}$ then
 $$\delta < m_x(\{x\}) =  f_{*}m_x(\{f(x)\}) \leq m_{f(x)} (\{f(x)\}).$$
 
 Thus $f(A_{\delta}) \subset A_{\delta}$, and by the ergodicity of $f$ we have that
% The relation $f(\mathcal F^c_{R(x)}) \supset \mathcal F^c_{R(f(x))}$ gives the invariance of the set $A_\delta$, that is $f(A_\delta) \subset A_\delta$. 
% By ergodicity 
$Vol(A_\delta)$ is  zero or one, for each $\delta \geq 0$. Note that $Vol(A_0)=1$ and $Vol(A_1)=0$. Let $\delta_0$ be the critical point for which $Vol(A_\delta)$ changes value, i.e, $\delta_0 = \sup \{\delta : Vol(A_{\delta}) = 1\}$. This means that all the atoms have weight $\delta_0$. Since $m_x$ is a probability we have an $\alpha_0:=1/\delta_0$ number of atoms as claimed.
\end{proof}

\begin{lemma}
There is a positive number $L_0$ such that on almost every center leaf there is a point $x$ such that $B_c(x,L_0)$ (ball centered on $x$ inside $\mathcal F^c(x)$ of size $L_0$) contains all the atoms of the leaf $\mathcal F^c_f(x)$.
 \end{lemma}
 \begin{proof}
  We divide the proof in two possible cases. The first one is the case where we have finite number of atoms on each center leaf and the second is that we have infinitely many atoms on each center leaf.\\
  
  {\noindent \bf Case one (Finite atoms):} Suppose we have a finite number of atoms on each center leaf. Given $L \in \mathbb R$, consider the following set
   $$B_L=\{x \; | \; B_c(x,L) \text{ contains all the atoms of } \mathcal F^c_x \}.$$
   For any large enough $L$ the set $B_L$ has positive volume. To prove that $Vol(B_L)=1$ for some large $L$, it is sufficient to prove that $f^{-n} (B_L) \subset B_L$ for some $n$. This is because, $f^{n}$ is ergodic. We do so by proving that $f^{-n} (B_c(x,L)) \subset B_c(f^{-n}(x), L)$, where $B_c(\xi,\rho)$ stands for the ball inside of the center leaf centered on $\xi$ and radius $\rho$ and $d_c$ is the distance on the center leaf coming from the metric when restricted to the center leaf.
   
   It is enough to prove that $f^{-n} (B_c(x,L)) \subset B_c(f^{-n}(x), L)$ on the lift, since the projection is locally an isometry. We now work on the lift, but we carry the same notation as it should not make any confusion. Let $K\in \mathbb R$ such that
   $$||h - Id||_{C^0} < K,$$

 And $Q$ coming from the quasi-isometry property of the center foliation.  Let $y, z \in B_c(x,L)$ be the extremals, i.e. $d_c(y,z)=2L$,

\begin{eqnarray*}
 d_c(f^{-n}(y), f^{-n}(z)) &\leq& Q \| f^{-n}(y) - f^{-n}(z) \|
 \leq Q ( \| hf^{-n}(y) - hf^{-n}(z)\|+ 2K) \\
				&=& Q ( \| A^{-n} h(y) - A^{-n}h(z)\| + 2K) \\
				&=& Q({{e^{-n\lambda^{wu}(A)}}} \|h(y) - h(z)\|  + 2K )\\
				&\leq& Q({{e^{-n\lambda^{wu}(A)}}}( \|y - z\| +2K) + 2K ) \\
				&\leq& Q({{e^{-n\lambda^{wu}(A)}}}( d_c(y,z) +2K) + 2K )				=Q({{e^{-n\lambda^{wu}(A)}}}( 2L +2K) + 2K )
 \end{eqnarray*}

  Since $K$ is fixed, first consider a large $L$ and then a large enough $n$ such that $$Q({{e^{-n\lambda^{wu}(A)}}}( 2L +2K) + 2K ) < 2L.$$
  
  For these choices of $n$ and $L$, we have $f^{-n} (B_c(x,L)) \subset B_c(f^{-n}(x), L)$.\\
  
   %Hence $d_c(f^{-1}(y), f^{-1}(z)) \leq 

  {\noindent \bf Case two (Infinite Atoms):} Suppose we have an infinitely many atoms on each center leaf.
   Let $\beta \in \mathbb R$ be a large number (for instance much bigger than $K Q$ where $K$ is the distance between $h$ and the identity map and $Q$ is the quasi isometric constant in the definition \ref{qi}). Since we have a finite number of $\widetilde R_i$, from the previous Lemma, we know that there is a number $\tau \in \mathbb R$ for which every center segment of size smaller then $\beta$ must contain at most $\tau$ atoms. But, since on each center leaf there are infinity many atoms, take a segment of leaf big enough so that it contains more then $\tau$ atoms. Iterate this segment backwards  and it will eventually be smaller than $\beta$ but containing more than $\tau$ atoms. Indeed,
   $$ h \circ f^{-n} = A^{-n} \circ h$$
   $$
   \|h (f^{-n}(x)) - h(f^{-n} (y)) \| = \| A^{-n} (h(x)) - A^{-n} (h(y))\| $$$$\leq e^{-n \lambda^{wu}(A)} \| h(x) -h(y) \|
   $$
  As $h$ is at a distance $K$ to identity we have 
  $$
  \|f^{-n}(x) - f^{-n}(y)\| \leq e^{-n \lambda^{wu}(A)} \| h(x) -h(y) \| + K \leq \frac{\beta}{Q}
  $$
  So, finally by quasi isometric property 
  $$
  d_c(f^{-n}(x), f^{-n}(y)) \leq \beta.
  $$

 The above contradiction implies that the number of atoms can not be infinite and now we proceed as in the previous case.

 \end{proof}

\begin{lemma}
The disintegration of the Lebesgue measure along the central leaves is mono-atomic, i.e there is just one atom per leaf.
\end{lemma}
 \begin{proof}
 We have a finite number of atoms on each center leaf and since the center foliation is an oriented foliation we may talk about the first atom. The set of first atom of all generic leaves is an invariant set with positive measure, therefore it has full measure. This means there is only one atom per center leaf, which concludes the proof.
 \end{proof}
  
  The above Lemma concludes the proof of the theorem.
\end{proof}

\begin{problem}
Is there any ergodic invariant measure $\mu$ with disintegration having more than one atom on leaves?  
\end{problem}

We note once again that since the work of Ponce-Tahzibi \cite{PT} assures that the set of DA satisfying the hypothesis of the next theorem is non-empty, we prove that these diffeomorphisms have atomic disintegration. 

\subsection{A Glimpse of Pesin Theory}
Before presenting the proof of Theorem $\ref{theorem:existatom}$, we recall some basic notions of Pesin theory. Let $f : \mathbb T^3 \rightarrow \mathbb T^3$ a partially hyperbolic diffeomorphism with splitting
$$T M = E^s \oplus E^c \oplus E^u.$$
Call $\Gamma$ the set of regular points of $f$, that is, the set of points $x \in \mathbb T^3$ for which the Lyapunov exponents are well defined. Then, for each $x \in \Gamma$ we define the Pesin-stable manifold of $f$ at $x$ as the set
$$W^s(x) = \left\{ y: \limsup_{n\rightarrow \infty} \frac{1}{n} d(f^n(x),f^n(y)) < 0 \right\}.$$

The Pesin-stable manifold is an immersed sub manifold of $\mathbb T^3$. Similarly we define the Pesin-unstable manifold at $x$, $W^u(x)$, using $f^{-1}$ instead of $f$ in the definition. 

It is clear that for a partially hyperbolic diffeomorphism $W^s(x)$ contains the stable leaf $\mathcal{F}^s(x).$ In  Theorem \ref{theorem:existatom} we assume that the central Lyapunov exponent is negative and consequently the Pesin-stable manifolds are two dimensional. By $\mathcal W^c(x)$ we denote the intersection of the Pesin stable manifold  $W^s(x)$ of $x$ with the center leaf $\mathcal F_f^c(x)$ of $x$. These manifolds depends only measurable on the base point $x,$ as it is proved in the Pesin theory. However, there is a filtration of the set of regular points by Pesin blocks: $\Gamma = \bigcup_{l \in \mathbb{N}} \Gamma_{l}$ such that each $\Gamma_l$ is a closed (not necessarily invariant) subset and $x \rightarrow \mathcal W^c(x)$ varies continuously on each $\Gamma_l.$

%An important fact about the Pesin-stable manifolds is that in our context, the semi-conjugacy between $f$ and the linear Anosov $A$ sends whole Pesin-stable manifolds to points. This comes from the following property of the semi-conjugacy $h$: if $\tilde{h}$ is the lift of $h$ then $\tilde{h}(x) = \tilde{h}(y)$ if, and only if, there exists a constant $ C > 0$ such that 
%$$d(\tilde{f}^n(x),\tilde{f}^n(y)) < C, \forall n \in \mathbb Z.$$ 
%

\begin{maintheorema}\label{theorem:existatom}
Let $f: \mathbb T^3 \rightarrow \mathbb T^3$ be a volume preserving, DA diffeomorphism. Suppose its linearization $A$ has the splitting $T_AM = E^{su} \oplus E^{wu} \oplus E^s$ (su and wu represents strong unstable and weak unstable bundles.) If $f$ has $\lambda ^{c}(x) <0$ for Lebesgue almost every point $x \in \mathbb T^3$, then volume has atomic disintegration on $\mathcal F^{c}_f$, in fact it is one atom per center leaf.
\end{maintheorema}

\begin{proof} 

To begin, we prove that the size of the weak stable manifolds $\mathcal{W}^c(x)$ is uniformly bounded for $x$ belonging to the regular set. In particular this enables us to prove that the partition (mod-$0$) by $\mathcal{W}^c(x)$ is a measurable partition.

\begin{lemma}
 The size of  $\{ \mathcal W^{c}(x) \}_{ \{x \; : \; \lambda^c(x) <0\} }$ is uniformly bounded for $x \in \Gamma$. More precisely, the image of  $ \mathcal W^{c}(x)$ by $h$ is a unique point.
\end{lemma}

\begin{proof}
Let $\tilde{f}:\mathbb R^3 \rightarrow \mathbb R^3$ and $\tilde{A}:\mathbb R^3 \rightarrow \mathbb R^3$ denote the lifts of $f$ and $A$ respectively and $\tilde{h}:\mathbb R^3 \rightarrow \mathbb R^3$ the lift of the semi-conjugacy $h$ between $f$ and $A$. Consider $\gamma \subset \widetilde{\mathcal W}^c(x)$, where $\widetilde{\mathcal W}^c(x)$ is the lift of $\mathcal W^c(x)$. Thus, $\gamma$ is inside the intersection of the center manifold of $\tilde{f}$ and the Pesin-stable manifold of $\tilde{f}$ passing through $x$. 

Let us show that $\tilde{h}$ collapses $ \widetilde{\mathcal W}^{c}(x)$ to a unique point. If we prove that, it clearly comes out ( from the bounded distance of $h$ to identity that, the size of $ \mathcal W^{c}(x)$ is uniformly bounded. Suppose by contradiction that $h(\gamma)$ has more than one point. By semi-conjugacy $\tilde{h}(f^ n(\gamma)) = A^ n(\tilde{h}(\gamma)).$ As $\tilde{h}(\gamma)$ is a subset of weak unstable foliation of $\tilde{A}$ for large $n$ the size of $\tilde{A}^n(h(\gamma))$ is large. On the other hand, $\gamma$ is in the Pesin stable manifold of $f$ and consequently for large $n$, the size of $\tilde{f}^n(\gamma)$ is very small. As $\|\tilde{h} - id\| \leq K$ we conclude that for large $n$ the size of $\tilde{h}(\tilde{f}^n(\gamma))$ can not be very big. This contradiction completes the proof.
%Denote $L = diam^c(\gamma)$.
%
%Consider $K\in \mathbb R$ such that
%$$||\tilde{h} - Id||_{C^0} < K$$
%
%\noindent and recall that $A$ is linear and $\tilde{h}$ maps center leaves to center leaves, so for all $n \in \mathbb N$ we have
%$$diam (\tilde{A}^n(\tilde{h}(\gamma))) = (\lambda^c_A)^n \cdot diam (\tilde{h}(\gamma)).$$
%
%Now, lets evaluate the size of $diam^c(\tilde{f}^n(\gamma))$.
%\begin{equation*}
%\begin{split}
%  (\lambda^c_A)^n \cdot diam (\tilde{h}(\gamma)) &= diam (\tilde{A}^n(\tilde{h}(\gamma))) \leq diam^c(\tilde{h}(\tilde{f}^n(\gamma)))\\
%&= d^c(\tilde{h}(\tilde{f}^n(a)),\tilde{h}(\tilde{f}^n(b)) \leq Q \cdot d(\tilde{h}(\tilde{f} ^n(a)),\tilde{h}(\tilde{f}^n(b)) \\ 
%& \leq Q \cdot (2K + d(\tilde{f}^n(a),\tilde{f}^n(b))) \leq Q \cdot (2K + diam^c(\tilde{f}^n(\gamma)).
%\end{split}
%\end{equation*}
%
%So
%$$diam^c(\tilde{f}^n(\gamma)) \geq \frac{(\lambda^c_A)^n}{Q} \cdot (diam(\gamma) - 2K) -2K \geq \frac{(\lambda^c_A)^n}{Q} \cdot (L - 2K) -2K.$$
%
%Thus, if $L > 2K$ there exist an $n_0 \in \mathbb N$ such that for all $n \geq n_0$ we have
%$$diam^c(\tilde{f}^n(\gamma)) > L$$
%
%contradicting the fact that $\gamma$ was taken inside the Pesin-stable manifold of $\tilde{f}$ passing through $x$. Therefore, all elements $\mathcal W^c(x)$ of our partition have uniformly bounded length. 
\begin{corollary}
The family $\{ \mathcal W^{c}(x) \}_{ \{x \; : \; \lambda^c(x) <0\} }$ forms a measurable partition. 
\end{corollary}
This corollary uses the same idea of the proof of the Proposition \ref{prop:compact}. However, that proposition is proved for continuous foliations and we adapt the proof for the Pesin measurable lamination.

First of all we consider a new partition $\{ \overline{\mathcal W^c(x)} \}$ whose elements are the closure of the elements $\mathcal W^c(x)$, that is, $\overline{\mathcal W^c(x)}$ is a bounded length center segment with its extremum points. Since $h$ collapses the Pesin-stable manifolds of $f$ into points, two different elements $\mathcal W^c(x)$ and $\mathcal W^c(y)$ cannot have a common extrema, so that $\{\overline{\mathcal W^c(x)}\}$ is indeed a partition with compact elements. Let us prove that it is indeed a measurable partition.

Let $\{x_j\}_{j \in \mathbb{N}}$ be a countable dense set of $M = \mathbb{T}^3.$ For each $x_j$ and $k, l \in \mathbb{N}$ we define $C_l (x_j, k)$ to be the union of $\mathcal{W}^c(y),  y \in \Gamma_l$ such that $\mathcal{W}^c(y)$ intersects the closed ball $B(x_j, \frac{1}{k}).$ By continuity of $\mathcal{W}^c(.)$ on $\Gamma_l$ we conclude that $C_l (x_j, k)$ is closed and consequently measurable. Indeed, if $y_n \in C_l (x_j, k)$ converges to $y$ then $y \in \Gamma_l$ and moreover $\mathcal{W}^c(y)$ intersects the closure of $B(x_j, \frac{1}{k}).$ 

Now, we need to separate two weak stable manifolds by means of some $C_l (x_j, k).$ Taking two elements $\mathcal{W}^c(a)$ and  $\mathcal{W}^c(b)$ there exists $l \in \mathbb{N}$ such that $a , b \in \Gamma_l$ and it is enough to take small enough $k$ and some $x_j$ such that $C_l (x_j, k)$ contains $\mathcal{W}^c(a)$ and not $\mathcal{W}^c(b).$ It is easy to see that for each $x \in \Gamma$ we have $\mathcal{W}^c(x) = \bigcap_{k, l, j}C_l (x_j, k)^{*} $ where $C_l (x_j, k)^{*}$ is either $C_l (x_j, k)$ or $\mathbb{T}^3 \setminus C_l (x_j, k).$

Now, observe that if $y$ is an extremum point of $\mathcal W^c(x)$ then $y$ cannot have negative center Lyapunov exponent, otherwise it would be in the interior of $\mathcal W^c(x)$. So, the set of extremum points of the elements $\overline{W^c(x)}$ is inside the set of points with non negative center Lyapunov exponent, and therefore it has zero measure. Thus, removing such points, the measurability of the partition $\{\overline{\mathcal W^c(x)}\}$ implies the measurability of $\{\mathcal W^c(x)\}$, concluding the proof of the lemma.
%
%Since $h$ is from a a $C^0$ bounded  \marginpar{G: I think that quasi-isometry is implicit here. Maybe we should open it a little bit more} distance from the identity, we know that there is a positive number $L$ for which, any center segment of size $L$ is sent by $f$ to a center segment of size greater then $L$. This means that $W^c(x)$ has length bounded by $L$. Hence, since it has bounded size it forms a measurable partition.
\end{proof}

%Consider $\pi: \mathbb T^3 \rightarrow \mathbb T^3 / \{ \mathcal W^{c}(x) \}$ the natural projection. And let $\nu := \pi_* Vol$. Also consider the map $ \widetilde f: \mathbb T^3 / \{ \mathcal W^{c}(x) \} \rightarrow \mathbb T^3 / \{ \mathcal W^{c}(x) \}$ \marginpar{G: This $\tilde{f}$ is not used in what follows!!} naturally obtained from $f$, which preserves the measure $\nu$.

\begin{lemma}\label{lemma:gamma}
 There exist a set $\Lambda \subset \mathbb T^3$ and a real number $R > 0$,  such that $Vol (\Lambda) > 0.5$ and if $x \in \Lambda$, then $diam^c(\mathcal W^{c}(x)) > R$.
\end{lemma}
\begin{proof}
Comes from the fact that $\lim_{n \rightarrow \infty} Vol (\{x : diam^c (\mathcal W^c(x)) > 1/n\}) = 1$
\end{proof}

The next lemma is inspired on the work of Ruelle, Wilkinson \cite{RW}.

\begin{lemma}
Disintegration of volume on the measurable partition $ \{ \mathcal W^{c}(x) \}$ is atomic.
\end{lemma}

\begin{proof}
Let $\pi: \mathbb T^3 \rightarrow \mathbb T^3 / \{ \mathcal W^{c}(x) \}$ be the natural projection, $\nu := \pi_* Vol$ and $\Lambda$ as in Lemma \ref{lemma:gamma}. Consider $B= \pi (\Lambda)$ and $N$ be the  minimum number of balls with diameter $R/10$ needed to cover $\mathbb T^3$. We also denote $\{\eta_x\}$ the system of conditional measures of $Vol$ on the partition $\{\mathcal W^c(x)\}$ defined previously.
%Pesin center-stable manifold of $x$. 
For $x \in \mathbb T^3$ define $$m(x) = inf \sum diam^c (U_i \cap \mathcal F^c_ f(x))$$
where the infimum is taken over all collections of closed balls $U_1, \ldots, U_k$ in $\mathbb T^3$ such that $k \leq N$ and $\eta_x (\cup_{i=1}^k U_j) \geq 0.5$.

We now define $$m=\mbox{ess} \sup_{x \in B \;}m(x).$$

We claim that $m=0$. Suppose, by contradiction that $m>0$. Then, given $\varepsilon>0$ there exist an integer $J$ such that
$$C \varepsilon^J N < m/2,$$ where $C>0$ is such that $d^c(f^i(y), f^i(z))\leq C \varepsilon ^i d^c(y,z)$ for all pair of points $y,z \in \Lambda$ with $z \in \mathcal W^c(y)$.
%which lie in the same Pesin center-stable manifold.

Let $\mathcal U$ be a cover of $\mathbb T^3$ by $N$ closed balls of diameter $R/10$. For $x \in \mathbb T^3$ such that $\pi(x)\in B$, let $U_1(x), \ldots, U_{k(x)}(x)$ such that $\eta_x(\cup_{i=1}^{k(x)}U_i(x)) \geq 0.5$. Note that

$$f^j_*\eta_x = \eta_{f^j(x)} \Rightarrow \eta_{f^j(x)}\left (\bigcup_{i=1}^{k(x)}f^j(U_i(x))\right) \geq 0.5. \; \forall i \in \mathbb N.$$

Also note that $ diam^c (f^j(U_i(x))) \leq \alpha \varepsilon^j.$

By Poincar\'e reccurence theorem, $y = f^{J_0}(x)$

\begin{eqnarray*}
 m(y) & \leq & \sum_{j=1}^{k(x)} diam (f^{J_0}(U_j(x))) \\
 &\leq & C k(x) \varepsilon^{J_0} \\
 & \leq & C N \alpha \varepsilon^{J_0} \\
 & < & m/2.
\end{eqnarray*}

Then, $m = ess \; sup_{x \in B}m(x) < m/2$, which is a contradiction with $m >0$.

Hence, $m=0$ implies that there is a sequence of closed balls with diameter going to zero and having measure greater then $0.5/N$. By Hammerlindl-Ures \cite{HU} we know that $f$ is ergodic, hence we have atomic disintegration.

\end{proof}

Once obtained atomicity we can apply Theorem \ref{theorem:oneatom} (see Remarks \ref{remark: nonnull} and \ref{remark:proof.theorem}) to get one atom per leaf.
\end{proof}

\section{Maximal entropy measure}

Given a volume preserving partially hyperbolic diffeomorphism $f: \mathbb T^3 \rightarrow \mathbb T^3$, we say that a measure $\mu$ is a {\bf maximizing entropy measure} if the metric entropy with respect to the measure is equal to the topological entropy of $f$, i.e,
$$h_{\mu}(f) = h_{top}(f).$$

Given a volume preserving partially hyperbolic diffeomorphism $f:\mathbb T^3 \rightarrow \mathbb T^3$ that is isotopic to a linear Anosov, the maximizing entropy measure is unique and is natural, in the sense that it is just the pull-back of the volume measure by the semi-conjugacy function $h$. More specifically, R. Ures showed in \cite{Ures}

\begin{theorem}
Let $f:\mathbb T^3 \rightarrow \mathbb T^3$ be an absolutely partially hyperbolic diffeomorphism homotopic to a hyperbolic linear automorphism $A:\mathbb T^3 \rightarrow \mathbb T^3$. Then, $f$ has a unique maximizing entropy measure $\mu_f$. More precisely, if $h$ is the semi-conjugacy between $f$ and $A$, i.e, $h \circ f = A \circ h$, then
$$m = h_{*}\mu_f$$

where $m$ denotes the Lebesgue measure on $\mathbb T^3$.
\end{theorem}

In the same article, R. Ures also showed that the center Lyapunov exponent for the maximizing measure entropy is greater or equal to the center exponent of the linear hyperbolic diffeomorphism. In the hypothesis of Theorem \ref{theorem:existatom} this fact contrasts with what happens to the center exponent with respect to the Lebesgue measure, and with this we prove in Theorem \ref{ugibbs} that $\mu_f$ cannot be $u$-gibbs.

\begin{theorem}[Theorem $5.1$ of \cite{Ures}] \label{theorem:raul}
Let $f:\mathbb T^3 \rightarrow \mathbb T^3$ be a $C^{1+\alpha}$ absolutely partially hyperbolic diffeomorphism homotopic to a hyperbolic linear automorphism $A$ with center Lyapunov exponent $\lambda^c_A > 0$. Let $\mu_f$ be the maximizing measure of $f$. Then, the center Lyapunov exponent of $\mu_f$, $\lambda^c_{\mu_f}$, satisfies
$$\lambda^c_{\mu_f} \geq \lambda^c_A.$$
\end{theorem}

To prove Theorem \ref{ugibbs} we need a celebrated result of Ledrappier-Young \cite{LY1}.

\begin{theorem}[Theorem $A$ of \cite{LY1}] \label{LY}
Let $f: M \rightarrow M$ be a $C^2$ diffeomorphism of a compact Riemannian manifold $M$ preserving a Borel probability measure $\mu$. Then $\mu$ has absolutely continuous conditional measure on unstable manifolds if, and only if, 
$$h_{\mu}(f) = \int \sum_{i} \lambda^{+}_i(x) \dim E_i(x) d\mu(x)$$

where $a^{+}:= \max\{a,0\}$.
\end{theorem}

\begin{maintheorema}\label{ugibbs}
Let $f : \mathbb T^3 \rightarrow \mathbb T^3$ be a partially hyperbolic diffeomorphism with Anosov linearization $A: \mathbb T^3 \rightarrow \mathbb T^3$. Assume that $\lambda^c_A>0$ and $\lambda^c_f < \lambda^c_A$. Then $\mu_f$ is not $u$-gibbs. Also denoting by $\lambda^u_{\mu_f}, \lambda^c_{\mu_f}, \lambda^s_{\mu_f}$ the Lyapunov exponents of $f$ (for $\mu_f$ almost every point) then
$$\lambda^c_{\mu_f} > \lambda^c_A \text{ or } \lambda^u_{\mu_f} > \lambda^u_A .$$
\end{maintheorema}
\begin{proof}
By contradiction, assume that $\mu_f$ is $u$-Gibbs. Consider
$$A^{+}:=\{x\in M : \lambda^c(x) \text{ is defined and } \lambda^c(x) \geq \lambda^c_A \}.$$

By \ref{theorem:raul} we have that $\mu_f(A^{+}) = 1$. Since $\mu_f$ is $u$-gibbs, for some $x\in \mathbb T^3$ (actually for $\mu_f$ almost every $x$) the center unstable leaf $\mathcal F^{cu}(x)$ intersects the set $A^{+}$ in a set of positive leaf measure, that is,
$$m_{cu}(\mathcal F^{cu}(x) \cap A^{+}) >0.$$

Now, define $B=\mathcal F^{cu}(x) \cap A^{+}$ and set
$$C= \bigcup_{y \in B} \mathcal F^s.$$

Then, for all $y\in C$ we have that $\lambda^c(y) \geq \lambda^c_A$. Now by absolute continuity property, we get that 
$$m(C)>0.$$

That is, we constructed a set of positive Lebesgue measure for which every point has center Lyapunov exponent bigger then $\lambda^c_A$ contradicting one of the hypothesis. So $\mu_f$ is not $u$-gibbs as we wanted to show.

Now, from Theorem \ref{LY}, given an invariant measure $\mu$ we can write
$$h_{\mu}(f) = \lambda^u_{\mu}\gamma_1 + \lambda^c_{\mu}\gamma_2$$

for some constants $0< \gamma_1,\gamma_2 \leq 1$ where
$$\gamma_1=\gamma_2=1 \Leftrightarrow \mu_f \text{ is }u-\text{gibbs}.$$

Also, we know that $h_{top}(f) = h_{top}(A)$. Then, since $\mu_f$ is the maximizing entropy measure it follows that
\begin{equation} \label{desig}
\lambda^u_{\mu_f}\gamma_1 + \lambda^c_{\mu_f}\gamma_2 = \lambda^u_A + \lambda^c_A.
\end{equation}

%By \ref{center} we have
%$$\lambda^u + \lambda^c  \geq \lambda^u_{\mu}\gamma_1 + \lambda^c \gamma_2.$$

Since $\mu_f$ is not $u$-Gibbs, then $\gamma_1$ and $\gamma_2$ cannot be both $1$. So
$$\lambda^u_{\mu_f} + \lambda^c_{\mu_f}  > \lambda^u_{\mu_f}\gamma_1 + \lambda^c_{\mu_f} \gamma_2 = \lambda^u + \lambda^c.$$

So either $\lambda^c_{\mu}>\lambda^c$ or $\lambda^u_{\mu} > \lambda^u.$
\end{proof}

We end by analyzing the disintegration of the measure of maximal entropy.

\begin{maintheorema}\label{theorem:atomic.max.mesure}
Let $f:\mathbb T^3 \rightarrow \mathbb  T^3$ be a partially hyperbolic diffeomorphism with Anosov linearization. If the disintegration of the maximizing entropy measure $\mu_f$ on the central leaves is atomic then it has exactly one atom per leaf.
\end{maintheorema}
\begin{proof}
The proof is analogous to the proof of Theorem \ref{theorem:oneatom}, in this case we have to verify that
\begin{itemize}
\item $\mu_f$ is ergodic and
 
\item the $\mu_f(\cup_i \partial_c \widetilde R_i)=0$ (where $\partial_c \widetilde R_i$ as defined in the proof of Theorem \ref{theorem:oneatom}).
\end{itemize}

The first item was already observed by R. Ures in \cite{Ures}. %(in his article, R. Ures observes that more then ergodic, the system $(f,\mu_f)$ is Bernoulli). 
The second item comes from the definition of $\mu_f$. We know that the entropy maximizing measure $\mu_f$ is unique and that $h_{*}\mu_f = m$. Now, by the definition of the sets $\widetilde{R}_i = h^{-1}(R_i)$ we have that $\partial_c \widetilde{R}_i = h^{-1}(\partial_c R_i)$. Thus
$$\mu_f(\partial_c \widetilde{R}_i ) = \mu_f (h^{-1}(\partial_c R_i)) = h_{*}\mu_f(\partial_c R_i) = m(\partial_c R_i) =0.$$

The proof follows as in the proof of Theorem \ref{theorem:oneatom}.
\end{proof}

%\begin{corollary} \marginpar{Ali: What???}
%With the same hypothesis of Theorem \ref{ugibbs}, if $E^{cs}_f(x) = E^{cs}_A(x)$ for a.e $x\in \mathbb T^3$ then $\lambda^c_{\mu_f}> \lambda^c$.
%\end{corollary}

%\begin{proof}
%In the case $E^{cs}_f(x)=E^{cs}_A(x)$ then 
%
%$$\lambda^c_f+\lambda^s_f = \int \log |\det ( Df(x)|E^{cs}_f(x) )| d\mu = \int \log |\det (DA(x)|E^{cs}_A(x))| d\mu = \lambda^s+\lambda^c$$
%
%where the last equality occurs because $\det (A|E^{cs}_A(x)) = \lambda^c \cdot \lambda^s$, for every $x\in \mathbb T^3$.
%
%In particular, 
%$$0=\int \log |\det Df(x)| d\mu =\int \log |\det (Df(x)|E^u_f(x)) |d\mu + \int \log |\det (Df(x)|E^{cs}_f(x))| d\mu = \lambda^u_f+\lambda^s+\lambda^c.$$
%
%Then we conclude that $\lambda^u_f = \lambda^u$, which implies $\lambda^c_f > \lambda^c$ as we wanted to show.
%\end{proof}

%With the results of this paper we proved that for the DA case ``everything" can happen with the disintegration of volume on the central leaves, that is, we do not have the dichotomy: absolutely continuous or atomic. The following problem address this question for the maximizing measure entropy.
%
%\begin{problem}
%Let $f: \mathbb T^3 \rightarrow \mathbb T^3$ be a partially hyperbolic diffeomorphism with Anosov linearization and consider $\mu_f$ to be its maximizing measure entropy. Is it true that the disintegration of $\mu_f$ on the central leaves is either atomic or absolutely continuous?
%\end{problem}

\bibliographystyle{plain}
\bibliography{Referencias}

\begin{thebibliography}{10}

\bibitem{AVW}
A.~Avila, M.~Viana, and A.~Wilkinson.
\newblock Absolute continuity, lyapunov exponents and rigidity i: geodesic
  flows.
\newblock {\em arXiv:1110.2365v2}, 2011.

\bibitem{BB}
A.~Baraviera and C.~Bonatti.
\newblock Removing zero lyapunov exponents.
\newblock {\em Ergodic Theory and Dynamical Systems}, pages 1655--1670, 2003.

\bibitem{BW}
K.~Burns and A.~Wilkinson.
\newblock On the ergodicity of partially hyperbolic systems.
\newblock {\em Annals of Mathematics}, 171(1):451--489, 2010.

\bibitem{BW05}
A.~Wilkinson C.~Bonatti.
\newblock Transitive partially hyperbolic diffeomorphisms on 3-manifolds.
\newblock {\em Topology}, 2005.

\bibitem{Franks}
J.~Franks.
\newblock Anosov diffeomorphisms.
\newblock {\em Global Analysis (Proc. Sympos. Pure Math., Vol. XIV, Berkeley,
  Calif., 1968) pp. 61-93 Amer. Math. Soc.}, 1970.

\bibitem{Gogolev1}
A.~Gogolev.
\newblock How typical are pathological foliations in partially hyperbolic
  dynamics: an example.
\newblock {\em Israel Journal of Mathematics}, 187(1):493--502, 2012.

\bibitem{PT}
G.Ponce and A.Tahzibi.
\newblock Zero center lyapunov exponent and non-compact center leaves.
\newblock {\em to Appear in the Proceedings of AMS}, 2013.

\bibitem{Htese}
A.~Hammerlindl.
\newblock Leaf conjugacies on the torus.
\newblock {\em Ph.D. Thesis}, 2009.

\bibitem{HU}
A.~Hammerlindl and R.~Ures.
\newblock Ergodicity and partial hyperbolicity on the 3-torus.
\newblock {\em Preprint. arXiv:1208.5660}, 2012.

\bibitem{HHU1}
F.R. Hertz, J.R. Hertz, and R.~Ures.
\newblock Accessibility and stable ergodicity for partially hyperbolic
  diffeomorphisms with 1d-center bundle.
\newblock {\em Inventiones Mathematicae}, 172:353--381, 2008.

\bibitem{HP}
M.~Hirayama and Y.~Pesin.
\newblock Non-absolutely continuous foliations.
\newblock {\em Israel Journal of Mathematics}, 160:173--187, 2007.

\bibitem{LY1}
F.~Ledrappier and L.~S. Young.
\newblock The metric entropy of diffeomorphisms: Part i: Characterization of
  measures satisfying pesin's entropy formula.
\newblock {\em Annals of Mathematics}, 122(3):509--539, 1985.

\bibitem{BBI}
M.Brin, D.Burago, and S.Ivanov.
\newblock On partially hyperbolic diffeomorphisms of 3-manifolds with
  commutative fundamental group.
\newblock {\em Modern dynamical systems and applications}, pages 307--312,
  2004.

\bibitem{BBI09}
M.Brin, D.Burago, and S.Ivanov.
\newblock Dynamical coherence of partially hyperbolic diffeomorphisms of the
  3-torus.
\newblock {\em J. Mod. Dyn}, pages 1--11, 2009.

\bibitem{MT}
F.~Micena and A.~Tahzibi.
\newblock Regularity of foliations and lyapunov exponents for partially
  hyperbolic dynamics.
\newblock {\em Nonlinearity}, 23:1071--1082, 2013.

\bibitem{HP13}
A.~Hammerlindl R.~Potrie.
\newblock Pointwise partial hyperbolicity in 3-dimensional nilmanifolds.
\newblock {\em Preprint}, 2013.

\bibitem{RW}
D.~Ruelle and A.~Wilkinson.
\newblock Absolutely singular dynamical foliations.
\newblock {\em Comm. Math. Phys.}, 219:481--487, 2001.

\bibitem{SX}
Radu Saghin and Zhihong Xia.
\newblock Geometric expansion, {L}yapunov exponents and foliations.
\newblock {\em Ann. Inst. H. Poincar\'e Anal. Non Lin\'eaire}, 26(2):689--704,
  2009.

\bibitem{SW00}
M.~Shub and A.~Wilkinson.
\newblock Pathological foliations and removable zero exponents.
\newblock {\em Inventiones Mathematicaeb}, 2000.

\bibitem{Ures}
R.~Ures.
\newblock Intrinsic ergodicity of partially hyperbolic diffeomorphisms with a
  hyperbolic linear part.
\newblock {\em Proc. Amer. Math. Soc.}, 140(6), 2012.

\bibitem{Va}
R.~Var{\~a}o.
\newblock Center foliation: absolute continuity, disintegration and rigidity.
\newblock {\em arXiv:1302.1637}, 2013.

\end{thebibliography}

\end{document}